\documentclass[12pt]{amsart}
\usepackage{graphicx}
\usepackage{color,graphics,graphicx,shortvrb}
\usepackage{epsfig}
\usepackage{amssymb,amsmath,amsfonts}
\usepackage{newlfont}
\usepackage{caption}
\usepackage{latexsym}

\newtheorem{theo}{Theorem}

\newtheorem{rema}{Remark}
\newtheorem{prop}{Proposition}

\def \d {\displaystyle}

\def \barr {\begin{array}{l}}
\def \ear {\end{array}}
\def \beq {\begin{equation}}
\def \eeq {\end{equation}}
\def \beqn {\begin{eqnarray}}
\def \eeqn {\end{eqnarray}}
\def \f {\end{document}}

\usepackage{amsfonts}
\def\dfrac{\displaystyle\frac}

\newcommand{\field}[1]{\mathbb{#1}}
\newcommand{\R}{\field{R}}

\newcommand{\N}{\field{N}}
\begin{document}
\title[Delayed wave equation without displacement term]{\bf Note on exponential and polynomial convergence for a delayed wave equation without displacement}
\author{Ka\"{\i}s Ammari}
\address{UR Analysis and Control of PDEs, UR13ES64, Department of Mathematics, Faculty of Sciences of Monastir, University of Monastir, 5019 Monastir, Tunisia}
\email{kais.ammari@fsm.rnu.tn}

\author{Boumedi\`ene Chentouf}
\address{Kuwait University, Faculty of Science, Department of Mathematics, Safat 13060, Kuwait}
\email{chenboum@hotmail.com, boumediene.chentouf@ku.edu.kw}

\begin{abstract}

This note places primary emphasis on improving the asymptotic behavior of a multi-dimensional delayed wave equation in the absence of any displacement  term. In the first instance, the delay is assumed to occur in the boundary. Then, invoking  a geometric condition \cite{ba,LR} on the domain, the exponential convergence of solutions to their equilibrium state is proved. The strategy adopted of the proof is based on an interpolation inequality combined with a resolvent method. In turn, an internal delayed wave equation is considered in the second case, where the domain possesses trapped ray and hence (BLR) geometric condition does not hold. In such a situation  polynomial convergence results are established. These finding improve earlier results of \cite{amchen,phung}.
\end{abstract}

\subjclass[2010]{34B05, 34D05, 70J25, 93D15}
\keywords{wave equation, time-delay, asymptotic behavior, exponential  convergence, polynomial convergence}

\maketitle
\tableofcontents

\thispagestyle{empty}

\section{Introduction}
Given a natural number $n \geq 2$, consider an open bounded connected set of  $\Omega$ in $\mathbb{R} ^n$, with a sufficiently smooth boundary $\Gamma=\partial \Omega$. We assume that $(\Gamma_0, \Gamma_1)$ is a partition  of $\Gamma$. The first system, we treat in the present paper, is the wave equation
\beq
y_{tt} (x,t) - \Delta y(x,t)=0, \; \hspace{3cm} \mbox{in} \;\; \Omega \times (0,\infty),
\label{1}
\eeq
together with a boundary damping
\begin{equation}
\left \lbrace
\begin{array}{ll}
\frac{\textstyle \partial y }{\textstyle \partial \nu } (x,t) = 0, & \mbox{on} \, \Gamma_0 \times (0,\infty),\\[2mm]
\frac{\textstyle \partial y }{\textstyle \partial \nu } (x,t) = -\alpha y_t (x,t)-\beta y_t (x,t-\tau), & \mbox{on} \, \Gamma_1 \times (0,\infty),
\label{n2}
\end{array}
\right.
\end{equation}
and the initial conditions
\begin{equation}
\left \lbrace
\begin{array}{ll}
y(x,0)=y_0(x), \; y_t(x,0)=z_0(x), & x \in \Omega,\\
y_t(x,t)=f(x,t),  & (x,t) \in \Gamma_1 \times (-\tau,0),
\label{2}
\end{array}
\right.
\end{equation}
where $\nu$ is the unit normal of $\Gamma$ pointing towards the exterior of $\Omega$, $\alpha >0$, and for sake for simplicity $ \beta >0$. Assuming in the first lieu that $\Gamma_1$ is nonempty while $\Gamma_0$ may be empty, we shall invoke the geometric condition \cite{ba,LR} in the ``control zone`` $\Gamma_1$ in order to show that the solutions to the above system exponentially converge to an equilibrium state. The strategy adopted is to use an interpolation inequality combined with a resolvent method.

In the second lieu, we study the asymptotic behavior of a wave equation in a three-dimensional domain with a trapped ray and hence the geometric control condition is violated. Specifically, consider the following delayed wave equation with an internal (distributed) localized damping:
\beq
y_{tt} (x,t) - \Delta y(x,t) + a(x) y_t(x,t) + b(x) y_t(x,t-\tau) =0, \; \hspace{3cm} \mbox{in} \;\; \Omega \times (0,\infty),
\label{1int}
\eeq
as well as the following boundary and initial conditions
\begin{equation}
\left \lbrace
\begin{array}{ll}
\frac{\textstyle \partial y }{\textstyle \partial \nu } (x,t) = 0, & \mbox{on} \, \Gamma \times (0,\infty),\\[1mm]
y(x,0)=y_0(x), \; y_t(x,0)=z_0(x), & x \in \Omega,\\
y_t(x,t)=g(x,t),  & (x,t) \in \Omega \times (-\tau,0),
\label{2int}
\end{array}
\right.
\end{equation}
in which $a$ is a non-negative function in $L^\infty(\Omega)$ and depends on a non-empty proper subset $\omega$ of $\Omega$ on which $1/a \in L^\infty(\omega)$ and in particular, $\left\{x \in \Omega; a(x) > 0\right\}$ is a non-empty open set of $\Omega$, and $b \in L^\infty(\Omega)$ is a non-negative function. Assuming, as in \cite{phung}, the trapping geometry on $\Omega$, we are able to show that the solutions converge to their equilibrium state with a polynomial decay despite the presence of the delay.

\medskip

%%%%%This result extends that of 

The outcomes of this work improve those available in literature in a number of directions: firstly, we are able to extend the result \cite{amchen}, where only the logarithmic convergence has been established. Secondly, the polynomial decay rate of \cite{phung} is extended to the case where an internal delay occurs in the wave equation. Lastly, it is show that the optimal result obtained in \cite{stahn} remains valid despite the presence a delay term.

Now, let us briefly outline the content of this article. In Section \ref{sect2}, preliminaries and problem statement are presented. Section \ref{sect3} deals with the exponential convergence of solutions to an equilibrium state under the geometric control condition (BLR) on the stabilization zone $\Gamma_1$. In section \ref{sect4}, we prove a polynomial convergence result for (\ref{1int})-(\ref{2int}) with a trapping geometric condition. Finally, this note ends with a conclusion.

\section{Preliminaries} \label{sect2}
\setcounter{equation}{0}

\subsection{Delayed boundary damping}
Consider the state space
$$ {\mathcal H}=H^1(\Omega) \times L^2(\Omega) \times L^2(\Gamma_1 \times (0,1)),$$
equipped with the inner product
\beq
\barr
\langle (y,z,u), (\tilde y,\tilde z, \tilde u) \rangle_{\scriptscriptstyle {\mathcal H}} = \d \int_{\Omega} \left( \nabla y \nabla \tilde y+z \tilde z \right)  dx + \xi \int_0^1 \int_{\Gamma_1} u \tilde u \, d\sigma d\rho  \; + \\
 \d \varpi \left[   \int_{\Omega} z \, dx  + (\alpha + \beta) \int_{\Gamma_1} y \, d\sigma -\beta \tau \int_0^1 \int_{\Gamma_1} u  \, d\sigma d\rho  \right] \left[  \int_{\Omega} \tilde z \, dx  + (\alpha + \beta) \int_{\Gamma_1}  \tilde y \, d\sigma -\beta \tau \int_0^1 \int_{\Gamma_1}  \tilde u \, d\sigma d\rho  \right], \label{5}
\ear
\eeq
where $\varpi >0$ is a positive such that
\beq
\varpi < \min \left \{ \frac{1}{(\alpha + \beta)(\alpha + \beta-\delta)}, \frac{\delta}{2(\alpha + \beta-\delta) 
\left| \Omega \right|} , \frac{\delta \xi}{2(\alpha + \beta-\delta) \left| \Gamma_1 \right| } \right\}.
\label{var}
\eeq

Furthermore, we assume that $\alpha, \, \beta$ and $\xi$ obey the following conditions
\beq
\barr
0 < \beta < \alpha,\\
\tau \beta < \xi < \tau (2\alpha-\beta).
\label{xx}
\ear
\eeq
Thereby, it has been shown in \cite[Proposition 1]{amchen} that the space ${\mathcal H}$ equipped with the inner product (\ref{5}) is a Hilbert space provided that $\varpi$ satisfies (\ref{var}). Moreover, letting $z=y_t$, and
$u(x,\rho,t)=y_t(x,t-\tau \rho ), \, \quad \quad x \in \Gamma_1, \, \rho \in (0 ,1), \,   t>0,$ \cite{da}, the system (\ref{1})-(\ref{2}) can be formulated  in an abstract differential equation in the Hilbert state space $ {\mathcal H}$  as follows
\beq
\left\{
\barr
\Phi_t (t)= {\mathcal A} \Phi (t),\\
\Phi (0)= \Phi_0 =(y_0,z_0,f),
\ear
\right.
\label{s}
\eeq
in which ${\mathcal A}$ is the unbounded linear operator defined by
\beq
\barr
 {\mathcal D}({\mathcal A}) =\Bigl\{ (y, z,u ) \in H^1(\Omega) \times H^1(\Omega) \times L^2(\Gamma_1; H^1 (0,1)); \Delta y \in L^2(\Omega); \,
\frac{\textstyle \partial y }{\textstyle \partial \nu } = 0 \, \mbox{on} \, \Gamma_0; \\
\hspace{5.2cm} \frac{\textstyle \partial y }{\textstyle \partial \nu }+\alpha z + \beta u(\cdot,1)=0, \, \mbox{and} \, \; z=u(\cdot,0) \, \mbox{on} \, \Gamma_1
\label{10}
\Bigr \},
\ear
\eeq
and
\beq
{\mathcal A} (y, z,u ) = ( z, \Delta y, -\tau^{-1} u_{\rho} ), \quad  \forall (y, z,u ) \in {\mathcal D}({\mathcal A}).
\label{11}
\eeq

We have the following result:
\begin{prop} (\cite[Proposition 2]{amchen})
Under the conditions (\ref{xx}), the linear operator ${\mathcal A}$ generates a $C_0$-semigroup of contractions $S(t)$ on ${\mathcal H}=\overline {{\mathcal D}({\mathcal A})}$. Additionally, for any initial data $\Phi_0 \in {\mathcal D}({\mathcal A})$, the system (\ref{s}) possesses a unique strong solution $\Phi (t) = S(t) \Phi_0 \in {\mathcal D}({\mathcal A})$ for all $ t \geq 0$ such that
$ \Phi (\cdot)  \in C^1 (\R^+;{\mathcal H}) \cap C  ( \R^+;{\mathcal D} ({\mathcal A}))$. In turn, if $\Phi_0 \in {\mathcal H}$, then the system (\ref{s}) has a unique weak solution $\Phi (t)= S(t) \Phi_0 \in {\mathcal H}$ such that $ \Phi (\cdot)  \in C^0 (\R^+;{\mathcal H} )$.
\label{l1}
\end{prop}

\subsection{Delayed internal damping}

Let
$$
L^2_0 (\Omega) = \left\{u \in L^2(\Omega); \, \int_\Omega u(x) \, dx = 0 \right\}, \quad H^{1,0}(\Omega) = H^1(\Omega) \cap L^2_0(\Omega),
$$
and the state space
$$
\mathcal{X} = H^{1,0} (\Omega) \times L^2_0 (\Omega) \times L^2 (\Omega \times (0,1)),
$$
endowed with the inner product
\begin{equation}
\label{5int}
\langle(y_1,z_1,u_1),(y_2,z_2,u_2)\rangle_{\mathcal{X}} = \int_\Omega \left(\nabla y_1 \nabla y_2 + z_1 z_2 \right) \, dx + \xi \, \int_0^1 \int_\Omega u_1 u_2 \, dx \, d\rho.
\end{equation}

\medskip

Additionally, $a,b$ and $\xi$ are supposed to verify the following conditions:
\begin{equation}
\label{xxint}
0 < \left\|b\right\|_{L^\infty(\Omega)}  <  \left\|a\right\|_{L^\infty(\Omega)}, \quad \tau  \left\|b\right\|_{L^\infty(\Omega)} < \xi < \tau \left(2 \left\|a\right\|_{L^\infty(\Omega)} -  \left\|b\right\|_{L^\infty(\Omega)}\right).
\end{equation}

Thanks to  the change of variables $u(x,\rho,t)=y_t(x,t-\tau \rho ), \,\, (x, \rho ,t) \in \Omega \times  (0 ,1) \times (\infty,0)  $ \cite{da}, the system (\ref{1int})-(\ref{2int}) can be written as follows
\beq
\left\{
\barr
\Phi_t (t)= {\mathcal A}_{a,b} \Phi (t),\\
\Phi (0)= \Phi_0 =(y_0,z_0,g),
\ear
\right.
\label{sint}
\eeq
where ${\mathcal A}_{a,b}$ is an unbounded linear operator defined by
\beq
\barr
 {\mathcal D}({\mathcal A}_{a,b}) =\Bigl\{ (y, z,u ) \in H^{1,0}(\Omega) \times H^{1,0}(\Omega) \times L^2(\Omega; H^1 (0,1)); \Delta y \in L^2(\Omega); \,
\frac{\textstyle \partial y }{\textstyle \partial \nu } = 0 \, \mbox{on} \, \Gamma; z=u(\cdot,0) \, \mbox{on} \, \Omega
\label{10int}
\Bigr \},
\ear
\eeq
and
\beq
{\mathcal A}_{a,b} (y, z,u ) = ( z, \Delta y - az - bu, -\tau^{-1} u_{\rho} ), \quad  \forall (y, z,u ) \in {\mathcal D}({\mathcal A}_{a,b}).
\label{11int}
\eeq

Assume that the conditions (\ref{xxint}) hold. Then, according to \cite{valein,ka3}, the linear operator ${\mathcal A}_{a,b}$ generates a $C_0$-semigroup of contractions $e^{t {\mathcal{A}}_{a,b}}$ on ${\mathcal X}=\overline {{\mathcal D}({\mathcal A}_{a,b})}$.

\section{Exponential convergence} \label{sect3}
\setcounter{equation}{0}
According to \cite{amchen}, we  know that we have the following strong asymptotic stability result for the system (\ref{s}).
\begin{theo} (\cite[Theorem 2]{amchen})
Assume that the conditions (\ref{xx}) hold. Given an initial data $\Phi_0=(y_0,z_0,f) \in {\mathcal H}$, we define $$\chi =\displaystyle  \dfrac{1}{ (\alpha+\beta) \left| \Gamma_1 \right| }  \left( \int_{\Omega} z_0 \, dx + (\alpha+\beta) \int_{\Gamma_1}  y_0 \, d\sigma -\beta \tau \int_0^1 \int_{\Gamma_1} f \, d\sigma d\rho  \right).$$
Then,  the unique solution  $\Phi (t)=(y(\cdot,t), y_t (\cdot,t)),y_t (\cdot,t-\tau \rho))$ of (\ref{s}) tends in ${\mathcal {H}}$ to $(\chi ,0,0)$, as $t \longrightarrow +\infty $ with a logarithmic decay $\log{(2 + t)}$.
\label{t1}
\end{theo}

In this section, we will improve the above result. Indeed, invoking the geometric condition (BLR) on the control zone $\Gamma_1$, we shall show that the convergence is actually exponential.

\medskip

To proceed, let $\dot{\mathcal{H}}$ be the closed subspace of of co-dimension $1$ of the state space $\mathcal{H}$ defined as follows:
$$
\dot{\mathcal{H}} = \left\{(y,z,u) \in \mathcal{H}; \,
\int_\Omega z(x) \, dx - \beta \tau \, \int_0^1 \int_{\Gamma_1} u(\sigma,\rho) \, d \sigma \, d \rho + (\alpha + \beta) \, \int_{\Gamma_1} y \, d \sigma = 0  \right\}
$$
and denote by $\dot{\mathcal{A}}$ the following new operator
$$
\dot{\mathcal{A}} : \mathcal{D}(\dot{\mathcal{A}}) := \mathcal{D}(\mathcal{A}) \cap \dot{\mathcal{H}} \subset \dot{\mathcal{H}} \rightarrow \dot{\mathcal{H}},
$$
\begin{equation}
\label{1.62bis}
\dot{\mathcal{A}} (y,z,u) = \mathcal{A} (y,z,u), \, \forall \, (y,z,u) \in \mathcal{D}(\dot{\mathcal{A}}).
\end{equation}

Under the conditions (\ref{xx}), the operator $\dot{\mathcal{A}}$ (see (\ref{11})) generates a $C_0$-semigroup of contractions $e^{t \dot{\mathcal{A}}}$ on $\dot{{\mathcal H}}$. Moreover, the spectrum $\sigma(\dot{\mathcal{A}})$ of $\dot{\mathcal{A}}$ consists of isolated eigenvalues of finite algebraic multiplicity only \cite{amchen}.

Recall the following frequency domain theorem for exponential stability from \cite{pruss,hung} of a $C_0$-semigroup of contractions on a Hilbert space:
\begin{theo}
\label{lemraokv}
Let $A$ be the generator of a $C_0$-semigroup  of contractions $S(t)$  on a Hilbert space $H$. Then, $S(t)$ is exponentially stable, i.e.,  for all $t >0$,
$$||S(t)||_{{\mathcal L}(H)} \leq C \, e^{-\omega t},$$
for some positive constants $C$ and $\omega $  if and only if
\begin{equation}
\rho (A) \supset \bigr\{i \gamma \bigm|\gamma \in \R \bigr\} \equiv i \R, \label{1.8wkv} \end{equation}
and
 \begin{equation}\limsup_{|\gamma| \to + \infty }  \|  (i\gamma I -A^{-1}\|_{{\mathcal L}({\mathcal X})} <\infty,
\label{1.9kv}
\end{equation}
where $\rho({A})$ denotes the resolvent set of the operator ${A}$.
\end{theo}

We are now in a position to state the first main result of this article:
\begin{theo} \label{lrkv}
Assume that the assumptions (\ref{xx}) hold. If $\Gamma_1$ satisfies a (BLR) geometrical control condition, then there exist $C, \omega >0$ such that
$$
\left\|e^{t \dot{\mathcal{A}}}\right\|_{{\mathcal L} (\dot{\mathcal{H}})} \leq C \, e^{-\omega t}, \;\; \forall t >0.
$$
In other words, the solutions of the system (\ref{1})-(\ref{2}) exponentially approach the equilibrium state $\chi$.
\end{theo}

\begin{proof}

Our first concern is to show that $i\gamma$  is not an eigenvalue of $\dot{\mathcal{A}}$ for any real number  $\gamma$, which clearly implies (\ref{1.8wkv}). To do so, it suffices to check that the only solution to the equation
\begin{equation}
\dot{\mathcal{A}} Z = i \gamma Z, \; Z= (y,z,u) \in \mathcal{D}(\dot{\mathcal{A}}), \; \gamma \in \mathbb{R},
\label{eigenkv}
\end{equation}
is the trivial solution. The proof of this desired result has been already obtained in \cite{amchen}.

Now, suppose that condition \eqref{1.9kv} does not hold. This gives rise, thanks to Banach-Steinhaus Theorem (see \cite{br}), to the existence of a sequence of real numbers $\gamma_n \rightarrow \infty$ and a sequence of vectors
$Z_n= (y_n,z_n,u_n) \in \mathcal{D}(\dot{\mathcal{A}})$ with $\|Z_n\|_{\dot{\mathcal{H}}} = 1$ such that
\begin{equation}
\|  (i \gamma_n I - \dot{\mathcal{A}})Z_n\|_{\dot{\mathcal{H}}} \rightarrow 0\;\;\;\; \mbox{as}\;\;\;n\rightarrow \infty,
\label{1.12kv} \end{equation}
i.e.,
\begin{equation} i \gamma_n y_n - z_n   \equiv  f_{n}\rightarrow 0 \;\;\; \mbox{in}\;\; H^1(\Omega),
\label{1.13kv}\end{equation}
\begin{equation}
i \gamma_n
z_n - \Delta y_n   \equiv   g_n \rightarrow 0 \;\;\;
\mbox{in}\;\; L^2(\Omega),
\label{1.13bkv} \end{equation}
 \begin{equation}
 i \gamma_n u_{n} + \frac{(u_n)_\rho}{\tau} \equiv  v_{n} \rightarrow 0 \;\;\;
\mbox{in}\;\; L^2(\Gamma_1 \times (0,1)).
\label{1.14bkv} \end{equation}

The ultimate outcome will be convergence of $\|Z_n\|_{\dot{\mathcal{H}}}$ to zero as $n\rightarrow \infty$, which contradicts the fact that $ \forall \ n  \in \N \,\, \left\|Z_n\right\|_{\dot{\mathcal{H}}}=1.$

Firstly, since
\begin{eqnarray*}
\left \|  (i \gamma_n I - \dot{\mathcal{A}})Z_n \right\|_{\dot{\mathcal{H}}} &\geq& \left| \Re \left(\langle (i\beta_n I - \dot{\mathcal{A}})Z_n, Z_n\rangle_{\dot{\mathcal{H}}} \right) \right| =- \Re \langle  \dot{\mathcal{A}})Z_n, Z_n \rangle_{\dot{\mathcal{H}}} \\
&\geq& \frac{1}{2} \left( (2\alpha-\beta -  \xi \tau^{-1} ) \int_{\Gamma_1} |z(\sigma)|^2 \, \d\sigma + (\xi \tau^{-1}-\beta) \, \int_{\Gamma_1}  |u(\sigma,1)|^2 \, d\sigma \right),
\end{eqnarray*}
it follows from \eqref{1.12kv} that
\begin{equation}
\label{cvz}
 z_{n} \rightarrow 0,  \;\; \textstyle{and} \;\;   u_n (\cdot,1) \rightarrow 0 \; \hbox{in} \; L^2(\Gamma_1).
\end{equation}
Therewith
\begin{equation}
\label{cvzz}
u_n (\cdot,0) \rightarrow 0 \; \hbox{in} \; L^2(\Gamma_1).
\end{equation}
%Then we have that
%\begin{equation}
%\label{cvnd}
%\frac{\partial y_n}{\partial \nu} = - \alpha z_n - \beta u_n %(\cdot,1) \rightarrow 0 \; \hbox{in} \; L^2(\Gamma_1).
%\end{equation}
Subsequently,  amalgamating (\ref{1.13kv}) and (\ref{cvz}), we get
\begin{equation}
\label{cvyt}
i \gamma_n \, y_n =  z_n +  f_n \rightarrow 0 \; \hbox{in} \; L^2(\Gamma_1),
\end{equation}
and hence
\begin{equation}
\label{cvy}
 y_n \rightarrow 0
\; \hbox{in} \; L^2(\Gamma_1).
\end{equation}

\medskip

On the other hand, \eqref{1.14bkv} yields
$$
u_n(x,\rho) = u_n(x,0) \, e^{- i \tau \gamma_n x} + \tau \, \int_0^\rho e^{- i \tau \gamma_n (\rho - s)} \, v_n (s) \, ds.
$$
Putting the above deduction together with \eqref{1.14bkv} and \eqref{cvzz}, we deduce that
\begin{equation}
\label{znkv}
u_n \rightarrow 0 \; \; \mbox{in}\;\; L^2(\Gamma_1 \times (0,1))\ .
\end{equation}

Now, let us  take the inner product of (\ref{1.13bkv}) with $z_n = i \gamma_n y_n - f_n$ in $L^2(\Omega)$ (see (\ref{cvyt})). A straightforward computation gives
$$
\int_\Omega |z_n|^2 \, dx - \int_\Omega |\nabla y_n |^2 \, dx  =
$$
\begin{equation}
\label{eqzy}
- \int_{\Gamma_1}
\frac{\partial y_n}{\partial \nu} \, \overline{y}_n \, d \sigma  + \frac{1}{i \gamma_n} \int_\Omega \nabla y_n \, \nabla \overline{f}_n \, dx   -
\frac{1}{i \gamma_n} \, \int_{\Gamma_1} \frac{\partial y_n}{\partial \nu} \, \overline{f}_n \, d\sigma + \frac{1}{i \gamma_n} \, \int_\Omega g_n \, \overline{z}_n \, dx  = \circ (1).
\end{equation}

The intention now is to evoke the (BLR) in our situation. To proceed, let $\dot{\mathcal{H}}_0$ be the closed subspace of $\mathcal{H}_0 := H^1(\Omega) \times L^2(\Omega)$ and of co-dimension $1$ given by
$$
\dot{\mathcal{H}}_0 = \left\{(y,z) \in \mathcal{H}_0; \,
\int_\Omega z(x) \, dx + \alpha  \, \int_{\Gamma_1} y \, d \sigma = 0  \right\},
$$
and then consider the operator $\dot{\mathcal{A}}_0$  defined by
$$
\dot{\mathcal{A}}_0 : \mathcal{D}(\dot{\mathcal{A}}_0) \subset \dot{\mathcal{H}}_0 \rightarrow \dot{\mathcal{H}}_0,
$$
and
$$
\mathcal{D}(\dot{\mathcal{A}}_0) = \left\{(u,v) \in \mathcal{H}_0; \, (v, \Delta u) \in \mathcal{H}_0, \, \frac{\partial u}{\partial \nu} = 0 \; \mbox{on} \; \Gamma_0, \, \frac{\partial u}{\partial \nu} + \alpha \, v = 0 \; \mbox{on} \; \Gamma_1 \right\} \cap \dot{\mathcal{H}}_0.
$$
\begin{equation}
\label{1.62bs}
\dot{\mathcal{A}}_0 (y,z) = (z,\Delta y), \, \forall \, (y,z) \in \mathcal{D}(\dot{\mathcal{A}}_0).
\end{equation}
In view of  the (BLR) geometric control condition \cite{ba}, we have 
\begin{equation}
\label{resol}
\limsup_{|\gamma| \to + \infty }  \|  (i\gamma I - \dot{{\mathcal A}}_0)^{-1}\|_{{\mathcal L}(\dot{{\mathcal H}}_0)} <\infty.
\end{equation}
Next, exploring the fact that $(y_n,z_n)$ satisfies of the following system:
$$
\left\{
\begin{array}{lll}
- \gamma_n^2 y_n - \Delta y_n = i \gamma_n f_n + g_n, \; \textstyle{on}\; \Omega,\\
 \frac{\partial y_n}{\partial \nu} = 0, \; \textstyle{on}\; \Gamma_0, \\
\medskip
 \frac{\partial y_n}{\partial \nu} + i \, \alpha \, \gamma_n \, y_n = \alpha \, f_n - \beta \, u_n (\cdot,1), \; \textstyle{on}\; \Gamma_1,
\end{array}
\right.
$$
it follows from (\ref{resol}) that there exists a positive constant $C$ such that for sufficiently large $n$, we have:
$$
\left\|y_n\right\|_{H^1(\Omega)} \leq C \, \left( \left\| f_n\right\|_{H^1 (\Omega)} + \left\|g_n \right\|_{L^2(\Omega)} +
\left\|u_n (\cdot,1) \right\|^2_{L^2(\Gamma_1)}  \right).
$$
Therewith, \eqref{1.13kv}, \eqref{1.13bkv} and \eqref{cvy} yield
\begin{equation}
\label{cvyg}
y_n \rightarrow 0 \; \hbox{in} \; H^1(\Omega).
\end{equation}
Combining \eqref{eqzy} and \eqref{cvyg}, we get
\begin{equation}
\label{zcvd}
z_n \rightarrow 0 \; \hbox{in} \; L^2(\Omega).
\end{equation}
In the light of \eqref{znkv}, \eqref{cvyg}  and  \eqref{zcvd}, we  conclude that $\left\| Z_n\right\|_{\dot{\mathcal{H}}} \rightarrow 0 $ which was our objective.

Lastly, the sufficient conditions of Theorem \ref{lemraokv} are fulfilled and the proof of Theorem \ref{lrkv} is completed.

\end{proof}

\section{Polynomial stability} \label{sect4}
This section is intended to establish a result that is between the previous one in Section \ref{sect3} (exponential convergence) and that of \cite{amchen} (logarithmic convergence) for the system (\ref{1int})-(\ref{2int}) (see also (\ref{sint})). This desirable outcome will be established by introducing, as in \cite{phung}, the trapping geometry on the domain $\Omega$ of $\mathbb{R}^3$. In fact, we are going to adopt most of the  notations in \cite{phung} and consider the same geometric situation as in \cite{phung} by letting $D(r_1, r_2) = \left\{(x_1, x_2) \in \R^2; |x_1| < r_1, |x_2| < r_2 \right\}$, where $r_1, r_2 > 0$. Moreover, we pick three positive constants $m_1$, $m_2$ and $\mu$. Then, we choose $\Omega$ a connected open set in $\mathbb{R}^3$ whose boundary is is $\partial \Omega := \Gamma = \Gamma_1 \cup  \Gamma_2 \cup Y$, where
$\Gamma_1 = \overline{D(m_1,m_2)} \times \left\{\mu \right\}$, with boundary $\partial \Gamma_1$,
$\Gamma_2 = \overline{D(m_1,m_2)} \times \left\{-\mu \right\}$, with boundary $\partial \Gamma_2,$ and $Y$ is a surface with boundary $\partial Y = \partial \Gamma_1 \cup  \partial \Gamma_2$. On the other hand, $\partial \Omega$ is assumed to be either $C^2$ with $Y  \subset (\R^2 \setminus D(m_1,m_2)) \times \R$ (in particular $Y \in C^2$) or $\Omega$ is convex (in particular $Υ$ is Lipschitz). Next, let $\Theta$ be a small neighborhood of $Y$ in $\mathbb{R}^3$ such that $\Theta \cap D(M_1,M_2) \times [-\mu,\mu]=\emptyset$ for some $M_1 \in (0,m_1)$ and $M_2 \in (0,m_2)$.  Lastly, we choose $\omega=\Omega \cap \Theta$.

\medskip
In such a situation of trapped geometry $(\Omega,\omega),$ the well-known geometric control condition (BLR) \cite{ba} is not fulfilled and hence  Phung has established in \cite{phung} the polynomial decay convergence result of (\ref{1int})-(\ref{2int}) {\bf without delay}, i.e, $b \equiv 0$:

\begin{theo}\cite[Phung]{phung} \label{phung}
There exist $C,\delta >0$ such that for all $t > 0$ we have:
$$
\left\|e^{t {\mathcal A}_{a,0}}\right\|_{{\mathcal L}(\mathcal{D}({\mathcal A}_{a,0}), {\mathcal X}_a)} \leq \frac{C}{t^{\delta}},
$$
\end{theo}
where
$$
{\mathcal A}_{a,0} :=
\left(
\begin{array}{llcc}
0 \; \; \; \; \; \;  I \\
\Delta \; -a
\end{array}
\right) : \mathcal{D}({\mathcal A}_{a,0}) \subset
\mathcal{X}_a := H^{1,0} (\Omega) \times L^2_0(\Omega) \rightarrow \mathcal{X}_a
$$
and
$$
{\mathcal D}({\mathcal A}_{a,0}) :=\Bigl\{ (y, z) \in H^{1,0}(\Omega) \times H^{1,0}(\Omega); \Delta y \in L^2(\Omega); \,
\frac{\textstyle \partial y }{\textstyle \partial \nu } = 0 \, \mbox{on} \, \Gamma
\Bigr \}.
$$

Our second main result is:

\begin{theo} \label{lrkvint}
Suppose that the assumption (\ref{xxint}) is fulfilled and $\omega$ satisfies the trapped geometrical condition. Then, there exists $C >0$ such that the semigroup generated by the operator ${\mathcal{A}}_{a,b}$ see ((\ref{10int})-(\ref{11int})) satisfies 
$$
\left\|e^{t {\mathcal{A}}_{a,b}}\right\|_{{\mathcal L} (\mathcal{D}({\mathcal{A}}_{a,b}),\mathcal{H})} \leq C/t^{\delta/2}, \quad \forall t > 0,
$$
where $\delta > 0$ is the same constant given by Theorem \ref{phung}.
\end{theo}

The frequency domain theorem for polynomial stability  \cite{tomilov,BD} of a 
$C_0$-semigroup of contractions on a Hilbert space will be used:

\begin{theo}
\label{lemraokvint}
A $C_0$-semigroup $e^{t{\mathcal P}}$ of contractions on a Hilbert space ${\mathcal X}$ satisfies, for all $t >0$,
$$||e^{t{\mathcal P}}||_{{\mathcal L}(\mathcal{D}(\mathcal{P}),{\mathcal X})} \leq C/t^{\beta},$$
for some constant $C, \beta >0$  if and only if
\begin{equation}
\rho ({\mathcal P})\supset \bigr\{i \gamma \bigm|\gamma \in \R \bigr\} \equiv i \R, \label{1.8wkvint} \end{equation}
and
 \begin{equation}\limsup_{|\gamma| \to + \infty }  \| |\gamma|^{- 1/\beta} (i\gamma I -{\mathcal P})^{-1}\|_{{\mathcal L}(\mathcal{D}(\mathcal{P}),{\mathcal X})} <\infty.
\label{1.9kvint}
\end{equation}
\end{theo}

\begin{proof}

Firstly, we need to validate (\ref{1.8wkvint}) for the operator ${\mathcal{A}}_{a,b}$ by showing that  $i\gamma \notin \sigma \left( {\mathcal{A}}_{a,b} \right)$, for any arbitrary real number $\gamma$. This can be done by verifying that the equation
\begin{equation}
{\mathcal{A}}_{a,b} Z = i \gamma Z
\label{eigenkvint}
\end{equation}
with $Z= (y,z,u) \in \mathcal{D}({\mathcal{A}}_{a,b})$ and $\gamma \in \mathbb{R}$ has only the trivial solution. Indeed, the equation \eqref{eigenkvint} writes
\begin{eqnarray}
&&z = i \gamma y \label{eigen1int}\\
&& \Delta y  - a z - b u(1) = i \gamma z, \label{eigen2int}\\
&&- \frac{u_\rho}{\tau}  = i \gamma u, \label{eigen3int}\\
&& \frac{\textstyle \partial y }{\textstyle \partial \nu } = 0 \,\,\, \mbox{on} \, \Gamma, \label{eigen5int}\\
&&z=u(\cdot,0) \,\,\, \mbox{on} \, \Omega.
\label{eigen6int}
\end{eqnarray}
If $\gamma = 0$, then  (\ref{eigen1int})-(\ref{eigen6int}) lead us to claim that the only solution of \eqref{eigenkvint} is the trivial one. In turn, if $\gamma \neq 0$, then by taking the inner product of (\ref{eigenkvint}) with $Z$ we get:
\begin{eqnarray}
0=2 \Re \left(\langle {\mathcal{A}}_{a,b} Z,Z \rangle_{{\mathcal{H}}} \right) &\leq& \displaystyle   \left[ \left\|b\right\|_{L^\infty(\Omega)} - 2\left\|a\right\|_{L^\infty(\Omega)} + \xi \tau^{-1} \right] \int_{\Omega} |z(\sigma)|^2 \, \d\sigma \nonumber \\
&+& \displaystyle \left[ \left\|b\right\|_{L^\infty(\Omega)} -\xi \tau^{-1} \right] \, \int_{\Omega}  |u(\sigma,1)|^2 \, d\sigma  \;\; \leq 0. \label{1.7kvint}
\end{eqnarray}
Whereupon, we have $z=u(\cdot,1) = 0$ on $\Omega$. Consequently, the only solution of \eqref{eigenkvint} is the trivial one.

It remains now to show that the resolvent operator of ${\mathcal{A}}_{a,b}$ satisfies the condition \eqref{1.9kvint} with $\beta = \delta /2$. Otherwise, Banach-Steinhaus Theorem (see \cite{br}) gives rise  the existence of  a sequence of real numbers $\gamma_n \rightarrow \infty$ and a sequence of vectors
$Z_n= (y_n,z_n,u_n) \in \mathcal{D}({\mathcal{A}}_{a,b})$ with $\|Z_n\|_{\mathcal{X}} = 1$ such that
\begin{equation}
\|  |\gamma_n|^{2/\delta} (i \gamma_n I - {\mathcal{A}}_{a,b})Z_n\|_{\mathcal{X}} \rightarrow 0\;\;\;\; \mbox{as}\;\;\;n\rightarrow \infty,
\label{1.12kvint} \end{equation}
that is,
\begin{equation} |\gamma_n|^{2/\delta} \left(i \gamma_n y_n - z_n\right)   \equiv  f_{n}\rightarrow 0 \;\;\; \mbox{in}\;\; H^{1,0}(\Omega),
\label{1.13kvint}\end{equation}
 \begin{equation}
  |\gamma_n|^{2/\delta} \left( i \gamma_n
z_n - \Delta y_n  + a z_n + b u_n(\cdot,1) \right) \equiv   g_n \rightarrow 0 \;\;\;
\mbox{in}\;\; L^2_0(\Omega),
\label{1.13bkvint} \end{equation}
 \begin{equation}
 |\gamma_n|^{2/\delta} \left(i \gamma_n u_{n} + \frac{(u_n)_\rho}{\tau} \right) \equiv  h_{n} \rightarrow 0 \;\;\;
\mbox{in}\;\; L^2(\Omega \times (0,1)).
\label{1.14bkvint} \end{equation}

The immediate task is to derive from \eqref{1.12kvint} that $\|Z_n\|_{\mathcal{X}} \rightarrow 0$, which  contradicts $  \left\|Z_n\right\|_{\mathcal{X}}=1,$ for all $ n  \in \N.$

\medskip

According to (\ref{1.7kvint}), we have
\begin{equation}
\label{lint1}
|\gamma_n|^{1/\delta} \, \sqrt{a} \, z_n \rightarrow 0,
|\gamma_n|^{1/\delta} \, \sqrt{b} \, u_n(\cdot,1) \rightarrow 0, \; \hbox{in} \; L^2(\Omega),
\end{equation}
which in turn yields
\begin{equation} |\gamma_n|^{1/\delta} \left(i \gamma_n y_n - z_n\right)   \equiv  \frac{f_{n}}{|\gamma_n|^{1/\delta}} \rightarrow 0 \;\;\; \mbox{in}\;\; H^{1,0}(\Omega),
\label{1.13kvintf}\end{equation}
 \begin{equation}
  |\gamma_n|^{1/\delta} \left( i \gamma_n
z_n - \Delta y_n  + a z_n  \right) \equiv   - \frac{b u_n(\cdot,1)}{|\gamma_n|^{1/\delta}} +\frac{g_n}{|\gamma_n|^{1/\delta}} \rightarrow 0 \;\;\;
\mbox{in}\;\; L_0^2(\Omega),
\label{intf}
\end{equation}
and we note that $(y_n,z_n) \in \mathcal{D}(\mathcal{A}_a)$. Based on Theorem \ref{phung}, we conclude that
$$y_n \rightarrow 0 \; \hbox{in} \; H^{1,0}(\Omega) \; \hbox{and} \; z_n \rightarrow 0 \; \hbox{in} \; L^2(\Omega).$$

\medskip
Furthermore, we have
$$
u_n(x,\rho) = u_n(x,0) e^{-i\tau \gamma_n x} + \tau \int_0^\rho e^{-i\tau \gamma_n (\rho-s)} \, \frac{h_n(s)}{|\gamma_n|^{2/\delta}} \, ds =
$$
$$
z_n(x) e^{-i\tau \gamma_n x} + \tau \int_0^\rho e^{-i\tau \gamma_n (\rho-s)} \, \frac{h_n(s)}{|\gamma_n|^{2/\delta}} \, ds \rightarrow 0 \; \hbox{in} \; L^2(\Omega \times (0,1)).
$$
Hence $\|Z_n\|_{\mathcal{X}} \rightarrow 0$, and so we managed to show that the conditions (\ref{1.8wkvint}) and (\ref{1.9kvint}) are fulfilled. This achieves the proof of Theorem \ref{lrkvint}.
\end{proof}

\begin{rema}
In the case where $\Omega = (0,1) \times (0,1)$ and
$$
a(x) = \left\{
\begin{array}{ll}
1, \, \forall \, x \in (0,\varepsilon) \times (0,1),\\
0, \, \text{elsewhere},
\end{array}
\right.,
$$
where $\varepsilon > 0$ is a constant, we have according to \cite{stahn} that $\delta = 2/3$ (which is the optimal decay rate). We obtain in this case that the polynomial decay rate for the delayed system (\ref{1int})-(\ref{2int}) is given by $t^{-1/3}$.
\label{rem1}
\end{rema}

\begin{rema}
Using Theorem \ref{lrkvint} and arguing as in \cite{amchen}, one can prove that the solutions to the system  (\ref{1int})-(\ref{2int}) polynomially converge (with a rate $t^{-\delta/2}$) to their equilibrium solution $(\zeta,0,0)$, where 
$$\zeta =\displaystyle  \frac{1}{ \displaystyle \int_{\Omega} (a+b) \, dx }  \left[ \int_{\Omega} \left( z_0 +(a+b) y_0  -b \tau \int_0^1  f(\rho) \, \rho  \right) \, dx \right].$$
Particularly, the convergence rate is $t^{-1/3}$ in the case of Remark \ref{rem1}.
\end{rema}

\section{Conclusion}

This paper has addressed the issue of improving the asymptotic behavior of a wave equation under the presence of a boundary or internal delay term. Assuming the absence of any presence of any displacement, the convergence rate is shown to be either exponential under (BLR) control geometric condition in the control zone or polynomial when a trapped ray occurs in the geometry of the domain.

\section*{Acknowledgment}
This work was supported and funded by Kuwait University, Research Project No. (SM04/17).

\end{document}